\documentclass[12pt]{amsart}

\usepackage{a4}
\usepackage{amssymb}
\usepackage{amsmath,amsthm}
\usepackage{moreverb}
\numberwithin{equation}{section}

\newtheorem{theorem}{Theorem}[section]
\newtheorem{corollary}[theorem]{Corollary}
\newtheorem{lemma}[theorem]{Lemma}

\theoremstyle{definition}
\newtheorem{defn}[theorem]{Definition}
\theoremstyle{remark}

\theoremstyle{remark}

\theoremstyle{remark}

\def\N{\mathbb{N}}
\def\M{\mathbf{M}}
\def\B{\mathbf{B}}

\def\P{\mathbf{P}}
\def\Y{\mathbf{Y}}
\def\X{\mathbf{X}}
\def\A{\mathbf{A}}

\def\Q{\mathbf{Q}}
\def\R{\mathbb{R}}
\def\D{\mathbf{D}}
\def\K{\mathbf{K}}

\def\f{\mathbf{f}}

\def\x{\mathbf{x}}
\def\y{\mathbf{y}}
\def\z{\mathbf{z}}

\title{An algorithm for semi-infinite polynomial optimization}

\thanks{This work was performed during a visit in November 2010 at CRM (Centre de Recerca Matematica), a Mathematics center from
UAB (Universidad Autonoma de Barcelona), and the author wishes to gratefully acknowledge financial support from CRM}

\author{J.B. Lasserre}

\address{J.B. Lasserre: LAAS-CNRS, 7 Avenue du Colonel Roche,
31077 Toulouse C\'{e}dex 4, France.}
\email{lasserre@laas.fr}

\date{}

\begin{document}

\begin{abstract}
We consider the semi-infinite optimization problem:
\[f^*:=\min_{\x\in\X}\:\{f(\x): g(\x,\y)\,\leq \,0,\:\forall\y\in\Y_\x\,\},\]
where $f,g$ are polynomials and $\X\subset\R^n$ as well as $\Y_\x\subset\R^p$, $\x\in\X$, are
compact basic semi-algebraic sets. To approximate $f^*$ we proceed in two steps.
First, we use the ``joint+marginal" approach
of the author \cite{lass-param} to approximate from above the function $\x\mapsto\Phi(\x)=\sup \{g(\x,\y): \y\in \Y_\x\}$ by a polynomial $\Phi_d\geq\Phi$,
of degree at most $2d$, with the strong property that $\Phi_d$ converges to $\Phi$ for the $L_1$-norm, as $d\to\infty$
(and in particular, almost  uniformly for some subsequence $(d_\ell)$, $\ell\in\N$).
Therefore, to approximate $f^*$ one may wish to solve the polynomial optimization problem 
$f^0_d=\min_{\x\in\X} \{f(\x): \Phi_d(\x)\leq0\}$ via a (by now standard) hierarchy of semidefinite relaxations, and 
for increasing values of $d$.
In practice $d$ is fixed, small,
and one relaxes the constraint $\Phi_d\leq0$ to $\Phi_d(\x)\leq\epsilon$ with $\epsilon>0$, allowing to change $\epsilon$ dynamically. As $d$ increases, the limit of the optimal value
$f^\epsilon_d$ is bounded above by $f^*+\epsilon$.
\end{abstract}


\maketitle

\section{Introduction}
Consider the semi-infinite optimization problem:
\begin{equation}
\label{defsemi-inf}
\P:\qquad f^*:=\min_{\x\in\X}\:\{f(\x): g(\x,\y)\,\leq \,0,\:\forall\y\in\Y_\x\,\},
\end{equation}
where $\X\subset\R^n$, $\Y_\x\subset\R^p$ for every $\x\in\X$, and
some functions $f:\R^n\to\R$, $g:\R^n\times\R^p:\to\R$.

Problem $\P$ is called a {\it semi-infinite} optimization problem because of
the infinitely many constraints $g(\x,\y)\leq0$ for all $\y\in\Y_\x$ (for each fixed $\x\in\X$). It has many applications
and particularly in robust control.

In full generality $\P$ is a very hard problem and most methods aiming
at computing (or at least approximating) $f^*$ use discretization to
overcome the difficult semi-infinite constraint $g(\x,\y)\leq 0$ for all $\y\in\Y_\x$.
Namely, in typical approaches where $\Y_\x\equiv\Y$ for all $\x\in\X$ 
(i.e. no dependence on $\x$), the set $\Y\subset\R^p$ is discretized on a finite grid and if
the resulting nonlinear programming problems are solved to global optimality,
then convergence to a global optimum of the semi-infinite problem occurs as the grid size vanishes (see e.g. the discussion and the many references in \cite{mitsos}).
Alternatively, in \cite{mitsos} the authors provide lower bounds on $f^*$
by discretizing $\Y$ and upper bounds via convex relaxations of the inner 
problem $\max_{\y\in\Y}\{g(\x,\y)\}\leq0$. In \cite{parpas} the authors also use a discretization scheme of $\Y$ but now combined with a hierarchy of {\it sum of squares} convex relaxations for solving to global optimality. 

\subsection*{Contribution.} 
We restrict ourselves to problem $\P$ where :

\begin{itemize}
\item $f,g$ are polynomials, and 
\item $\X\subset\R^n$ and $\Y_\x\subset\R^p,\,\x\in\X$, are compact basic semi-algebraic sets.
\end{itemize}
For instance many problems of robust control can be put in this framework; see e.g. their description
in \cite{calafiore}.
Then in this context we provide a numerical scheme whose novelty with respect to previous works is to avoid discretization of
the set $\Y_\x$. Instead  we use the ``joint+marginal" methodology for parametric polynomial optimization 
developed by the author in \cite{lass-param}, to provide a sequence of polynomials
$(\Phi_d)\subset\R[\x]$ (with degree $2d$, $d\in\N$) that approximate from above the function
$\Phi(\x):=\max_{\y}\,\{g(\x,\y)\,:\,\y\in\Y_\x\}$, and with the strong property that
if $d\to\infty$ then
$\Phi_d\to \Phi$ in the $L_1$-norm. (In particular, $\Phi_{d_\ell}\to\Phi$ almost uniformly on $\X$ for some subsequence
$(d_\ell)$, $\ell\in\N$.)
Then, ideally, one could solve the 
nested sequence of polynomial optimization problems:
\begin{equation}
\label{substitute}
\P_d:\quad f^*_d=\min\:\{f(\x)\::\: \Phi_d(\x)\,\leq\,0\:\},\qquad d=1,2,\ldots
\end{equation}
For fixed $d$,
one may approximate (and often solve exactly) (\ref{substitute}) by solving a hierarchy of semidefinite relaxations, as defined
in \cite{las01}. However, as the size $O(d^n)$ of these semidefinite relaxations increases very fast
with $d$, in practice one rather let $d$ be fixed, small, and relax the constraint
$\Phi_d(\x)\leq0$ to $\Phi_d(\x)\leq \epsilon$ for some scalar $\epsilon>0$ that one may adjust dynamically during the algorithm.
As $d$ increases, the resulting optimal value $f^\epsilon_d$ is bounded above by $f^*+\epsilon$. The approach is illustrated on a sample of small problems taken from the literature.

\section{Notation, definitions and preliminary results}
Let $\R[\x]$ (resp. $\R[\x,\y]$) denote the ring of real polynomials in the variables
$\x=(x_1,\ldots,x_n)$ (resp. $\x$ and $\y=(y_1,\ldots,y_p)$), whereas $\Sigma[\x]$ (resp. $\Sigma[\x,\y]$) denote 
its subset of sums of squares.

Let $\R[\y]_k\subset\R[\y]$ denote the vector space of real polynomials of degree at most $k$.
For every
$\alpha\in\N^n$ the notation $\x^\alpha$ stands for the monomial $x_1^{\alpha_1}\cdots x_n^{\alpha_n}$ and for every 
$d\in\N$, let $\N^{n}_d:=\{\alpha\in\N^n:\sum_j\alpha_j\leq d\}$ with cardinal $s(d)={n+d\choose n}$.
Similarly $\N^{p}_d:=\{\beta\in\N^p:\sum_j\beta_j\leq d\}$ with cardinal ${p+d\choose p}$.
A polynomial $f\in\R[\x]$ is written 
\[\x\mapsto f(\x)\,=\,\sum_{\alpha\in\N^n}\,f_\alpha\,\x^\alpha,\]
and $f$ can be identified with its vector of coefficients $\f=(f_\alpha)$ in the canonical basis.
For a real symmetric matrix $\A$ the notation $\A\succeq0$ stands for $\A$ is positive semidefinite.

A real sequence $\z=(z_\alpha)$, $\alpha\in\N^n$, has a {\it representing measure} if
there exists some finite Borel measure $\mu$ on $\R^n$ such that 
\[z_\alpha\,=\,\int_{\R^n}\x^\alpha\,d\mu(\x),\qquad\forall\,\alpha\in\N^n.\]

Given a real sequence $\z=(z_\alpha)$ define the linear functional $L_\z:\R[\x]\to\R$ by:
\[f\:(=\sum_\alpha f_\alpha\x^\alpha)\quad\mapsto L_\z(f)\,=\,\sum_{\alpha}f_\alpha\,z_\alpha,\qquad f\in\R[\x].\]
\subsection*{Moment matrix}
The {\it moment} matrix associated with a sequence
$\z=(z_\alpha)$, $\alpha\in\N^n$, is the real symmetric matrix $\M_d(\z)$ with rows and columns indexed by $\N^n_d$, and whose entry $(\alpha,\beta)$ is just $z_{\alpha+\beta}$, for every $\alpha,\beta\in\N^n_d$. 
If $\z$ has a representing measure $\mu$ then
$\M_d(\z)\succeq0$ because
\[\langle\f,\M_d(\z)\f\rangle\,=\,\int f^2\,d\mu\,\geq0,\qquad\forall \,\f\,\in\R^{s(d)}.\]
\subsection*{Localizing matrix}
With $\z$ as above and $g\in\R[\x]$ (with $g(\x)=\sum_\gamma g_\gamma\x^\gamma$), the {\it localizing} matrix associated with $\z$ 
and $g$ is the real symmetric matrix $\M_d(g\,\z)$ with rows and columns indexed by $\N^n_d$, and whose entry $(\alpha,\beta)$ is just $\sum_{\gamma}g_\gamma z_{\alpha+\beta+\gamma}$, for every $\alpha,\beta\in\N^n_d$.
If $\z$ has a representing measure $\mu$ whose support is 
contained in the set $\{\x\,:\,g(\x)\geq0\}$ then
$\M_d(g\,\z)\succeq0$ because
\[\langle\f,\M_d(g\,\z)\f\rangle\,=\,\int f^2\,g\,d\mu\,\geq0,\qquad\forall \,\f\,\in\R^{s(d)}.\]
\begin{defn}[Archimedean property]
\label{ass-1}
A set of polynomials $q_j\in \R[\x]$, $j=0,\ldots,p$ (with $q_0=1$), satisfy the Archimedean property if the quadratic polynomial $\x\mapsto M-\Vert\x\Vert^2$ can be written in the form:
\[M-\Vert\x\Vert^2=\sum_{j=0}^p\sigma_j(\x)\,q_j(\x),\]
for some sums of squares polynomials $(\sigma_j)\subset\Sigma[\x]$.
\end{defn}
Of course the Archimedean property implies that the set $\D:=\{\x\in\R^n\,:\,q_j(\x)\geq0,\:j=1,\ldots,p\}$ is compact. 
For instance, it holds whenever the level set $\{\x: q_k(\x)\geq0\}$ is compact for some $k\in\{1,\ldots,p\}$, or 
if the $q_j$'s are affine and $\D$ is compact (hence a polytope).
On the other hand, if $\D$ is compact then
$M-\Vert\x\vert^2\geq0$ for all $\x\in \D$ and some $M$ sufficiently large.
So if one adds the redundant quadratic constraint $\x\mapsto q_{p+1}(\x)=M-\Vert\x\Vert^2\geq0$ in the definition of $\D$ then the Archimedean property holds. Hence it is not a restrictive assumption.

Let $\D:=\{\x\in\R^n\,:\,q_j(\x)\geq0,\:j=1,\ldots,p\}$, and given a polynomial $h\in\R[\x]$, consider the hierarchy of semidefinite programs:
\begin{equation}
\label{genericsdp}
\quad\left\{\begin{array}{rll}
\rho_{\ell}=\displaystyle\min_\z&L_\z(h)&\\
\mbox{s.t.}&\M_\ell(\z),\M_{\ell-v_j}(q_j\,\z)&\succeq0,\quad j=1,\ldots,p,
\end{array}\right.
\end{equation}
where $\z=(z_\alpha)$, $\alpha\in\N^n_{2\ell}$, and $v_j=\lceil ({\rm deg}\,q_j)/2\rceil$, $j=1,\ldots,p$.
\begin{theorem}[\cite{las01,lassbook}]
\label{generic}
Let a family of polynomials $(q_j)\subset\R[\x]$ satisfy the Archimedean property.
Then as $\ell\to\infty$, $\rho_\ell\uparrow h^*=\min_\x\{h(\x)\,:\,\x\in\D\}$.
Moreover, if $\z^*$ is an optimal solution of (\ref{genericsdp}) and
\begin{equation}
\label{curto}
{\rm rank}\,\M_\ell(\z^*)\,=\,{\rm rank}\,\M_{\ell-v}(\z^*)\:(=:r)
\end{equation}
(where $v=\max_jv_j$) then $\rho_\ell=h^*$ and one may extract $r$ global minimizers $\x^*_k\in\D$, $k=1,\ldots,r$.
\end{theorem}
The size (resp. the number of variables) of the semidefinite program (\ref{genericsdp}) 
grows as $n+\ell\choose n$ (resp. as $n+2\ell\choose n$) and so becomes rapidly prohibitive, especially in view of the present status of available semidefinite solvers. Therefore, and even though 
practice reveals that convergence is fast and often finite, so far, the above methodology is limited to small to medium size problems (typically, and depending on the degree of the polynomials appearing in the data,
problems with up to $n\in[10,20]$ variables).
However, for larger size problems with sparsity in the data and/or symmetries, {\it adhoc} and tractable versions
of (\ref{genericsdp}) exist. See for instance the 
sparse version of (\ref{genericsdp}) proposed
in \cite{waki}, and whose convergence was proved in \cite{lass-sparse} when the sparsity pattern satifies the so-called {\it running intersection property}. In \cite{waki} this technique was shown to be successful on a sample
of non convex problems with up to $1000$ variables.

\section{Main result}
Let $\B\subset\R^n$ be a simple set like a box or an ellipsoid.
Let $p_s\in\R[\x]$, $s=1,\ldots,sx$, and 
$h_j\in\R[\x,\y]$, $j=1,\ldots,m$, be given polynomials and let $\X\subset\R^n$ be the basic semi-algebraic set
\[\X:=\{\x\in\R^n\::\: p_s(\x)\geq0,\quad s=1,\ldots,sx\}.\]
Next, for every $\x\in\R^n$, let $\Y_\x\subset\R^p$ be the basic semi-algebraic set described by:
\begin{equation}
\label{setyx}
\Y_\x\,=\,\{\,\y\in\R^p\::\:h_j(\x,\y)\,\geq\,0,\quad j=1,\ldots,m\,\},\end{equation}
and with $\B\supseteq\X$, let $\K\subset\R^n\times\R^p$ be the set
\begin{equation}
\label{setk}
\K\,:=\,\{(\x,\y)\,\in\,\R^{n+p}\::\:\x\in\B;\quad h_j(\x,\y)\geq0,\quad j=1,\ldots,m\}.
\end{equation}
Observe that problem $\P$ in (\ref{defsemi-inf}) is equivalent to:
\begin{eqnarray}
\label{equiv}
\P:\qquad f^*&=&\min_{\x\in\X}\:\{\,f(\x)\::\:\Phi(\x)\,\leq\,0\:\}\\
\label{defphi}
\mbox{where}\quad \Phi(\x)&=&\max_{\y}\,\{g(\x,\y)\::\:\y\in\Y_\x\,\},\quad\x\in\B.
\end{eqnarray}

\begin{lemma}
\label{lem1}
Let $\K\subset\R^{n+p}$ in (\ref{setk}) be compact
and assume that for every $\x\in\B\subset\R^n$, the set $\Y_\x$ defined in (\ref{setyx}) is nonempty. Then $\Phi$ is upper semicontinuous (u.s.c.) on $\B$.
Moreover, if there is some compact set $\Y\subset\R^p$ such that
$\Y_\x=\Y$ for every $\x\in\B$, then $\Phi$ is continuous on $\B$.
\end{lemma}
\begin{proof}
Let $\x_0\in\B$ be fixed, arbitrary, and let $(\x_k)_{k\in\N}\subset\B$ be a sequence 
that converges to $\x_0$  and such that
\[\limsup_{\x\to\x_0}\,\Phi(\x)\,=\,\lim_{k\to\infty} \Phi(\x_k).\]
As $\K$ is compact then so is $\Y_\x$ for every $\x\in\B$.
Therefore, as $\Y_\x\neq\emptyset$ for all $\x\in\B$
and $g$ is continuous, there exists an optimal solution $\y_k^*\in\Y_{\x_k}$ for every $k$.
By compactness there exist a subsequence $(k_\ell)$ and $\y^*\in\R^p$ such that
$(\x_{k_\ell},\y^*_{k_\ell})\to (\x_0,\y^*)\in\K$, as $\ell\to\infty$. Hence
\begin{eqnarray*}
\limsup_{\x\to\x_0}\,\Phi(\x)&=&\lim_{k\to\infty} \Phi(\x_k)\\
&=&\lim_{k\to\infty} g(\x_k,\y^*_k)\,=\,\lim_{\ell\to\infty} g(\x_{k_\ell},\y^*_{k_\ell})\\
&=&g(\x_0,\y^*)\,\leq\,\Phi(\x_0),
\end{eqnarray*}
which proves that $\Phi$ is u.s.c. at $\x_0$. As $\x_0\in\B$ was arbitrary,
$\Phi$ is u.s.c. on $\B$.

Next, assume that there is some compact set $\Y\subset\R^p$ such that
$\Y_\x=\Y$ for every $\x\in\B$. Let $\x_0\in\B$ be fixed arbitrary with
$\Phi(\x_0)=g(\x_0,\y_0^*)$ for some $\y_0^*\in\Y$. Let
$(\x_n)\subset\B$, $n\in\N$,  be a sequence such that $\x_n\to\x_0$ as $n\to\infty$,
and $\Phi(\x_0)\geq\displaystyle\liminf_{\x\to\x_0}\Phi(\x)=\displaystyle\lim_{n\to\infty}\Phi(\x_n)$.
Again, let $\y^*_n\in\Y$ be such that $\Phi(\x_n)=g(\x_n,\y^*_n)$, $n\in\N$.
By compactness, consider an arbitrary converging subsequence $(n_\ell)\subset \N$, i.e., such that
$(\x_{n_\ell},\y^*_{n_\ell})\to (\x_0,\y^*)\in\K$ as $\ell\to\infty$, for some $\y^*\in\Y$. Suppose that 
$\Phi(\x_0)\,(=g(\x_0,\y^*_0))\,>\,g(\x_0,\y^*)$, say $\Phi(\x_0)>g(\x_0,\y^*)+\delta$ for some $\delta>0$.
 By continuity of $g$, $g(\x_{n_\ell},\y^*_{n_\ell})<g(\x_0,\y^*)+\delta/2$ for every $\ell>\ell_1$ (for some $\ell_1$).
But again, by continuity, $\vert g(\x_{n_\ell},\y^*_0)-g(\x_0,\y^*_0)\vert <\delta/3$ whenever $\ell>\ell_2$ (for some $\ell_2$).
And so we obtain the contradiction 
\begin{eqnarray*}
\Phi(\x_{n_\ell})&\geq & g(\x_{n_\ell},\y^*_0)\,>\,\Phi(\x_0)-\delta/3\\
\Phi(\x_{n_\ell})&=&g(\x_{n_\ell},\y^*_{n_\ell})\,<\,\Phi(\x_0)-\delta/2,\end{eqnarray*}
whenever $\ell>\max[\ell_1,\ell_2]$. Therefore, $g(\x_0,\y^*_0)=g(\x_0,\y^*)$ and so,
\[g(\x_0,\y^*_0)\,=\,\Phi(\x_0)\,=\,g(\x_0,\y^*)\,=\,\displaystyle\lim_{\ell\to\infty} \Phi(\x_{n_\ell})\,=\,
\displaystyle\liminf_{\x\to\x_0} \Phi(\x)\leq\Phi(\x_0),\]
which combined with $\Phi$ being u.s.c., yields that
$\Phi$ is continuous at $\x_0$.
\end{proof}

We next explain how to
\begin{itemize}
\item approximate the function $\x\mapsto \Phi(\x)$ on $\B$ by a polynomial, and
\item evaluate (or at least approximate) $\Phi(\x)$ for some given $\x\in\B$, to check whether $\Phi(\x)\leq0$.
\end{itemize}
Indeed, these are the two main ingredients of the algorithm that we present later.

\subsection{Certificate of $\Phi(\x)\leq0$}

For every $\x\in\X$ fixed, let
$g_\x,h^\x_j\in\R[\y]$ be the polynomials
$\y\mapsto g_\x(\y)=g(\x,\y)$ and $\y\mapsto h^\x_j(\y):=h_j(\x,\y)$, $j=1,\ldots,m$,
and consider the hierarchy of semidefinite programs:
\begin{equation}
\label{fixed-x}
\Q_\ell(\x):\quad\left\{\begin{array}{rll}
\rho_{\ell}(\x)=\displaystyle\max_\z&L_\z(g_\x)&\\
\mbox{s.t.}&\M_\ell(\z),\M_{\ell-v_j}(h^\x_j\,\z)&\succeq0,\quad j=1,\ldots,m,
\end{array}\right.
\end{equation}
where $\z=(z_\beta)$, $\beta\in\N^p_{2\ell}$, and
$v_j=\lceil({\rm deg}\,h_j^\x)/2\rceil$, $j=1,\ldots,m$. Obviously one has $\rho_\ell(\x)\geq\Phi(\x)$ for every $\ell$, and
\begin{corollary}
\label{cor1}
Let $\x\in\X$ and assume that the polynomials $(h^\x_j)\subset\R[\y]$ satisfy the Archimedean property.
Then:

{\rm (a)} As $\ell\to\infty$, $\rho_\ell(\x)\downarrow \Phi(\x)=\max \{g(\x,\y)\,:\,\y\in\Y_\x\}$.  In particular,
if $\rho_\ell(\x)\leq0$ for some $\ell$, then $\Phi(\x)\leq0$. 

{\rm (b)} Moreover, if $\z^*$ is an optimal solution of (\ref{fixed-x}) that satisfies
\[{\rm rank}\,\M_\ell(\z^*)\,=\,{\rm rank}\,\M_{\ell-v}(\z^*)\:(=:r),\]
(where $v:=\max_jv_j$), then $\rho_\ell(\x)=\Phi(\x)$ and there are $r$ global maximizers $\y(k)\in\Y_\x$, $k=1,\ldots,r$.
\end{corollary}
Corollary \ref{cor1} is a direct consequence of Theorem \ref{generic}.

\subsection{Approximating the function $\Phi$}

Recall that $\B\supseteq\X$ is a {\it simple} set like e.g., a simplex, a box or an ellipsoid and
let $\mu$ be the finite Borel probability measure uniformly distributed on $\B$. 
Therefore, the vector $\gamma=(\gamma_\alpha)$, $\alpha\in\N^n$, of moments of $\mu$, i.e.,
\[\gamma_\alpha\,:=\,\int_\B\x^\alpha\,d\mu(\x),\qquad \alpha\in\N^n,\]
can be computed easily. For instance, in the sequel we assume that $\B=[-1,1]^n=\{\x:\,\theta_i(\x)\geq0,\,i=1,\ldots n\}$
with $\theta_i\in\R[\x,\y]$ being the polynomial $(\x,\y)\mapsto \theta_i(\x,\y):=1-x_i^2$, $i=1,\ldots,n$.

Observe that the function $\Phi$ is defined in (\ref{defphi}) via a {\it parametric} polynomial optimization problem
(with $\x$ being the parameter vector).  Therefore, following \cite{lass-param}, let $r_j=\lceil({\rm deg}\,h_j)/2\rceil$, $j=1,\ldots,m$, and consider the hierarchy of semidefinite relaxations indexed by $d\in\N$:
\begin{equation}
\label{primal}
\left\{\begin{array}{rll}
\rho_{d}=\displaystyle\max_\z&L_\z(g)&\\
\mbox{s.t.}&\M_d(\z),\M_{d-r_j}(h_j\,\z)&\succeq0,\quad j=1,\ldots,m\\
&\M_{d-1}(\theta_i\,\z)&\succeq0,\quad i=1,\ldots,n\\
&L_\z(\x^\alpha)&=\,\gamma_\alpha,\quad \alpha\in\N^n_{2d},
\end{array}\right.
\end{equation}
where the sequence $\z$ is now indexed in $\N^{n+p}_{2d}$, i.e.,
$\z=(z_{\alpha\beta})$, $(\alpha,\beta)\in\N^{n+p}_{2d}$.
Writing $g_0\equiv 1$, the dual of the semidefinite program (\ref{primal})  reads
\begin{equation}
\label{dual}
\left\{\begin{array}{rl}
\rho^*_{d}=\displaystyle\min_{q,\sigma_j,\theta_i}&\displaystyle\int_\B q(\x)\,d\mu(\x)\\
\mbox{s.t.}&q(\x)-g(\x,\y)=\displaystyle\sum_{j=0}^m\sigma_j(\x,\y)h_j(\x,\y)
+\displaystyle\sum_{i=1}^n\psi_i(\x,\y)\theta_i(\x,\y)
\\
&q\in\R[\x]_{2d},\:\sigma_j,\psi_i\in\Sigma[\x,\y]\\
&{\rm deg}\,\sigma_j\,h_j\,\leq 2d,\quad j=0,\ldots,m.\\
&{\rm deg}\,\psi_i\,\theta_i\,\leq 2d,\quad i=1,\ldots,n.
\end{array}\right.
\end{equation}
It turns out that any optimal solution of the semidefinite program (\ref{dual}) permits to approximate $\Phi$ 
in a strong sense.

\begin{theorem}[\cite{lass-param}]
\label{thm2}
Let $\K\subset\R^{n+p}$ in (\ref{setk}) be compact.
Assume that the polynomials $h_j,\theta_i\in\R[\x,\y]$ satisfy the Archimedean property
and assume that for every $\x\in\B$, the set $\Y_\x$ defined in (\ref{setyx}) is nonempty.
Let $\Phi_d\in\R[\x]_{2d}$ be an optimal solution of (\ref{dual}). Then :

{\rm (a)} $\Phi_d\geq\Phi$ and as $d\to\infty$,
\begin{equation}
\label{lem2-1}
\int_\B (\Phi_d(\x)-\Phi(\x))\,d\mu(\x)\,=\,
\int_\B \vert\, \Phi_d(\x)-\Phi(\x)\,\vert\,d\mu(\x)\:\to\:0,
\end{equation}
that is, $\Phi_d\to\Phi$ for the $L_1(\B,\mu)$-norm\footnote{$L_1(\B,\mu)$ is the Banach space of 
$\mu$-integrable functions on $\B$, with norm $\Vert f\Vert=\int_\B\vert f\vert d\mu$.}.

{\rm (b)} There is a subsequence $(d_\ell)$, $\ell\in\N$, such that $\Phi_{d_\ell}\to\Phi$, $\mu$-almost uniformly\footnote{If one fixes $\epsilon>0$ arbitrary then there is some $A\in\mathcal{B}(\B)$ such that
$\mu(A)<\epsilon$ and $\Phi_{d_\ell}\to\Phi$ uniformly on $\B\setminus A$, as $\ell\to\infty$.}  in $\B$, as $\ell\to\infty$.
\end{theorem}
The proof of (a) can be found in \cite{lass-param}, whereas (b) follows from (a) and 
\cite[Theorem 2.5.3]{ash}. \\

\subsection{An algorithm}

The idea behind the algorithm is to approximate $\P$ in (\ref{defsemi-inf}) with
the {\it polynomial} optimization problem:
$(\P^\epsilon_d)$:
\begin{equation}
\label{alternative}
\P^\epsilon_{d}:\quad f^\epsilon_{d}=\min_{\x\in\X}\:\{\,f(\x)\::\: \Phi_d(\x)\,\leq\,\epsilon\:\},\qquad d=1,2,\ldots
\end{equation}
with $d\in\N,\epsilon>0$ fixed, and $\Phi_d$ as in Theorem \ref{thm2},
for every $d=1,\ldots $.

Obviously, for $\epsilon=0$ one has $f^0_d\geq f^*$ for all $d$ because
by definition $\Phi_d\geq\Phi$ for every $d\in\N$. However, it may happen
that $\P^0_d$ has no solution. Next, if $\x^*$ is an optimal solution of $\P$
and $\Phi(\x^*)<0$, it may also happen that $\Phi_d(\x^*)>0$ if $d$ is not large enough. This is why one needs to relax
the constraint $\Phi\leq0$ to $\Phi_d\leq\epsilon$ for some $\epsilon >0$.
However, in view of Theorem \ref{thm2}, one expects that $f^\epsilon_d\approx f^*$ provided that
$d$ and $\epsilon$ are sufficiently large and small, respectively. And indeed:
\begin{theorem}
\label{th3}
Assume that $\X$ is the closure of an open set.
Let $\epsilon\geq0$ be fixed, arbitrary and with $f^\epsilon_d$ be as in (\ref{alternative}), 
let $\x^\epsilon_d\in\X$ be any optimal solution of (\ref{alternative}) (including the case 
where $\epsilon=0$), and let
\[\tilde{f}^\epsilon_d:=\min\{f^\epsilon_\ell:\,\ell=1,\ldots,d\}\,=\,f(\x^\epsilon_{\ell(d)})\quad\mbox{for some }\ell(d)\in\{1,\ldots,d\}.\]

{\rm (a)} If $\epsilon>0$ there exists $d_\epsilon\in\N$ such that for every $d\geq d_\epsilon$,
$f(\x^\epsilon_{\ell(d)})<f^*+\epsilon$.

{\rm (b)} If there is an optimal solution $\x^*\in\X$ of (\ref{defsemi-inf}) such that
$\Phi(\x^*)<0$, 
then there exists $d_0\in\N$ such that for every $d\geq d_0$,
$f^*\leq f(\x^0_{\ell(d)})<f^*+\epsilon$.
\end{theorem}
\begin{proof}
(a) With $\epsilon>0$ fixed, arbitrary, let $\x^*_\epsilon\in \X$ be such that 
$\Phi(\x^*_\epsilon)\leq0$ and $f(\x^*_\epsilon)<f^*+\epsilon/2$.
We may assume that $\x^*_\epsilon$ is not on the boundary of $\X$.
Let $O^1_\epsilon:=\{\x\in{\rm int}\,\X\,:\,\Phi(\x)<\epsilon/2\}$ which is an open set because 
$\Phi$ is u.s.c. (by Lemma \ref{lem1}), and so $\mu(O^1_\epsilon)>0$.
Next, as $f$ is continuous, there exists $\rho_0>0$ such that
$f<f^*+\epsilon$ whenever $\x\in O^2_\epsilon:=\{\x\in{\rm int}\,\X:\Vert \x-\x^*_\epsilon\Vert<\rho_0\}$.
Observe that $\rho:=\mu(O^1_\epsilon\cap O^2_\epsilon)>0$ because 
$O^1_\epsilon\cap O^2_\epsilon$ is an open set
(with $\x^*_\epsilon\in 
O^1_\epsilon\cap O^2_\epsilon$).
Next, by Theorem \ref{thm2}(b), there is a subsequence
$(d_\ell)$, $\ell\in\N$, such that $\Phi_{d_\ell}\to\Phi$, $\mu$-almost uniformly on $\B$.
Hence, there is some Borel set $A_\epsilon\subset\B$,
and integer $\ell_\epsilon\in\N$, such that $\mu(A_\epsilon) <\rho/2$ and 
$\displaystyle\sup_{\x\in\X\setminus A_\epsilon}\vert \Phi(\x)-\Phi_{d_\ell}(\x)\vert <\epsilon/2$ for all $\ell\geq \ell_\epsilon$. In particular, as $\mu(A_\epsilon)<\rho/2<\mu(O^1_\epsilon\cap O^2_\epsilon)$, the set 
$\Delta_\epsilon:=(O^1_\epsilon\cap O^2_\epsilon)\setminus A_\epsilon$ has positive $\mu$-measure.
Therefore,  $f(\x)<f^*+\epsilon$ and $\Phi_{d_\ell}(\x)<\epsilon$ whenever $\ell\geq \ell_\epsilon$ and $\x\in \Delta_\epsilon$, which in turn implies $f^\epsilon_{d_\ell}<f^*+\epsilon$, and consequently,
$\tilde{f}^\epsilon_d=f(\x^\epsilon_{\ell(d)})<f^*+\epsilon$, the desired result. 

(b) Let $\epsilon':=-\Phi(\x^*)$, and let $O^1_{\epsilon'}:=\{\x\in{\rm int}\,\X\,:\,\Phi(\x)<-\epsilon'/2\}$ which is a nonempty open set because it contains $\x^*$ and $\Phi$ is u.s.c.. Let $O^2_{\epsilon'}$ be as $O^2_\epsilon$
in the proof of (a), but now with $\x^*_{\epsilon'}=\x^*\in\X$. Both $O^1_{\epsilon'}$ and
$O^2_{\epsilon'}$ are open and nonempty because they contain $\x^*$. The rest of the proof is like for the proof of (a), but noticing that
now for every $\x\in\Delta_{\epsilon'}$ one has $\Phi_{d_\ell}(\x)<-\epsilon'/2+\epsilon'/2=0$, and so $\x$ is feasible for (\ref{alternative}) with $\epsilon=0$. Next,
by feasiblity $f(\x)\geq f^*$ since the resulting feasible set in (\ref{alternative}) is smaller than that
of (\ref{substitute}) because $\Phi_d\geq\Phi$, for all $d$.
And so $f^*\leq f(\x)<f^*+\epsilon$ whenever $\x\in\Delta_\epsilon$, and $\ell\geq \ell_\epsilon$,
from which (b) follows.
\end{proof}
Theorem \ref{th3} provides a rationale behind the algorithm that we present below.
In solving (\ref{alternative}) with $d$ sufficiently large
and small $\epsilon$ (or even $\epsilon=0$), 
$f^\epsilon_d$ would provide a good approximation of $f^*$. 
But in principle, computing the global optimum $f^\epsilon_d$ is still a difficult problem. However,
$\P^\epsilon_d$ is a polynomial optimization problem.
Therefore, by Theorem \ref{generic}, if the polynomials $(p_s)\subset\R[\x]$ that define $\X$ satisfy the Archimedean property (see Definition \ref{ass-1})
we can approximate $f^\epsilon_d$ from below, as closely as desired,
by a monotone sequence $(f^\epsilon_{dt})$, $t\in\N$, obtained
by solving the hierarchy of semidefinite relaxations (\ref{genericsdp}), 
which here read:
\begin{equation}
\label{fixed-x-QP}
\left\{\begin{array}{rll}
f^\epsilon_{dt}=\displaystyle\min_\z&L_\z(f)&\\
\mbox{s.t.}&\M_t(\z),\M_{t-d}(\epsilon-\Phi_d\,\z)&\succeq0\\
&\M_{t-t_s}(p_s\,\z)&\succeq0,\quad s=1,\ldots,sx,
\end{array}\right.
\end{equation}
where $t_s=\lceil{(\rm deg}\,p_s)/2\rceil$, $s=1,\ldots,sx$.
\begin{corollary}
Assume that the polynomials $(p_s)\subset\R[\x]$ satisfy the Archimedean property.
Then $f^\epsilon_{dt}\uparrow f^\epsilon_d$ as $t\to\infty$.
Moreover, if $\z^*$ is an optimal solution of 
(\ref{fixed-x-QP}) and
\begin{equation}
\label{curtobis}
{\rm rank}\,\M_t(\z^*)\,=\,{\rm rank}\,\M_{t-t_0}(\z^*)\:(=:r)
\end{equation}
(where $t_0:=\max[d,\max_s[t_s]]$) 
then $f^\epsilon_{dt}=f^\epsilon_d$ 
and one may extract $r$ global minimizers 
$\x^*_d(k)\in\X$, $k=1,\ldots,r$. That is, for every $k=1,\ldots,r$,  
$f(\x^*_d(k))=f^\epsilon_d$ and $\Phi_d(\x^*_d(k))\leq\epsilon$.
\end{corollary}
However, given a minimizer $\x^*_d\in\X$, if on the one hand $\Phi_d(\x^*_d)\leq\epsilon$,
on the other hand it may not satisfy $\Phi(\x^*_d)\leq0$.
(Recall that checking whether $\Phi(\x^*_d)\leq0$ can be done via solving
the hierarchy of relaxations $\Q_\ell(\x)$ in (\ref{fixed-x}) with $\x:=\x^*_d$.)
If this happens then one solves again (\ref{fixed-x-QP}) for a smaller value of $\epsilon$, etc., until 
one obtains some $\x^*_d\in\X$ with $\Phi(\x^*_d)\leq0$.

Finally, and as already mentioned, if $d$ is relatively large,
the size of semidefinite relaxations (\ref{fixed-x-QP}) to compute $f^\epsilon_{dt}$
becomes too large for practical implementation (as one must have $t\geq d$). So in practice
one let $d$ be fixed at a small value, typically the smallest possible value of $d$, i.e., $1$ ($\Phi_d$ is quadratic)
or $2$ ($\Phi_d$ is quartic)), and one updates $\epsilon$ as indicated above. So the resulting algorithm reads:
\vspace{0.5cm}

\subsection*{Algorithm}~\\ 
\noindent
{\bf Input:} $\ell,d,k^*\in\N$, $\epsilon_0>0$ (e.g. $\epsilon_0:=10^{-1}$), $d\in\N$, 
$\tilde{\x}:=\star$, $f(\star)=+\infty$.\\
\noindent
{\bf Output:} $f(\x^*_d)$ with $\x^*_d\in\X$ and $\Phi(\x^*_d)\leq0$.\\
Step 1: Set $k=1$ and $\epsilon(k)=1$. \\
Step 2: While $k\leq k^*$, solve $\P^{\epsilon(k)}_d$ in (\ref{alternative}) $\to \x^*_k\in\X$.\\
Step 3: Solve $\Q_\ell(\x^*_k)$ in (\ref{fixed-x}) $\to \rho_\ell(\x^*_k)$.\\
If $-\epsilon_0\leq\rho_\ell(\x^*_k)\leq0$ set 
$\x^*_d:=\x^*_k$ and STOP.\\
If $\rho_\ell(\x^*_k)<-\epsilon_0$ then:
\begin{itemize}
\item if $f(\tilde{\x})>f(\x^*_k)$ then set $\tilde{\x}:=\x^*_k$. If $k=k^*$ then $\x^*_d:=\x^*_k$.
\item set $\epsilon(k+1):=2\epsilon(k)$, $k:=k+1$ and go to Step 2.
\end{itemize}
If $\rho_\ell(\x^*_k)>0$ then:
\begin{itemize}
\item If $k< k^*$ set $\epsilon(k+1):=\epsilon(k)/2$, $k:=k+1$ and go to Step 2.
\item If $k= k^*$ then set set $\x^*_d=\tilde{\x}$.
\end{itemize}
\vspace{0.2cm}

Observe that in Step 2 of the above algorithm, one assumes that 
by solving $\P^{\epsilon(k)}_d$ one obtains $\x^*_k\in\X$.

\subsection{Numerical experiments}

We have taken Examples $2, 7, 9$, K, M, N, all from 
Bhattacharjee et al. \cite[Appendix A]{bhat} and whose data are polynomials, except for problem L.
For the latter problem, the non-polynomial function $\x\mapsto \min[0,(x_1-x_2)]$ is semi-algebraic and
can be generated by introducing an additional variable $x_3$, with 
the polynomial constraints:
\[x_3^2\,=\,(x_1-x_2)^2;\qquad x_3\geq0.\]
Indeed, $2\min[0,(x_1-x_2)]=x_1-x_2-x_3$.

Although these examples are quite small, they are still non trivial (and even difficult) to solve, and we wanted to test
the above methodology with small relaxation order  $d$. In fact we have even  considered the smallest possible $d$,
i.e., $d=1$ ($\Phi_d$ is quadratic).
Results in Table 1 are quite good since by using the semidefinite relaxation of
minimal order ``$d$" one obtains an optimal value $f^*_d$ quite close to $f^*$,
at the price of updating $\epsilon$ several times.

Next, for Problem L, if we now increase $d$ to $d=2$, we improve the optimal value
which becomes $f^*_d=0.3849$ with $\epsilon=2.2$.
However, for Problem M, increasing $d$ does not improve the optimal value.
\newpage

\begin{table}
\label{tab2}
\begin{center}
\begin{tabular}{||l| c | c | c |  c ||}
\hline
&best known value& $f^*_d$ & final value of $\epsilon$\\
\hline
\hline
 problem  2&0.194& 0.198 &1.895\\
 problem 7& 1.0& 1.41& 5\\
 problem 9 & -12.0&-14.47$^*$ &0\\
 problem K& -3.0&-3.0 &3.037\\
 problem L& 0.3431&0.435 &2.295\\
 problem M& 1.0&2.25 &2.592\\
problem N& 0.0&$10^{-8}$&0\\
\hline 
\end{tabular}
\end{center}
\caption{Examples of \cite[Table 6.1]{bhat} with minimal $d$}
\end{table}

\section{Conclusion}

We have presented an algorithm for semi-infinite (global) polynomial optimization 
whose novelty with respect to previous works is to {\it not} rely on a discretization scheme.
Instead, it uses a polynomial approximation $\Phi_d$ of the function $\Phi$, obtained by 
solving some semidefinite relaxation attached to the ``joint+marginal" approach developed
in \cite{lass-param} for parametric optimization, which guarantees (strong) convergence 
$\Phi_d\to\Phi$ in $L_1$-norm. Then for fixed $d$, one has to solve a polynomial
optimization problem, which can be done by solving an appropriate hierarchy of semidefinite relaxations.
Of course, as already mentioned and especially in view of the present status of semidefinite solvers,
so far the present methodology is limited to small to medium size problems, unless sparsity in the data
and/or symmetries are taken into account appropriately, as described in e.g. \cite{lass-sparse,waki}.
Preliminary results on non trivial (but small size) examples are encouraging.

\end{document}